\documentclass[notitlepage,11pt,reqno]{amsart}

\usepackage{amsopn,esint,nicefrac}
\usepackage{mathabx}
\usepackage[final]{hyperref}
\usepackage[T1]{fontenc}
\usepackage{graphicx}
\usepackage{cancel,pdfsync}
\usepackage{bbm}
\usepackage{tikz}
\usetikzlibrary{patterns} 
\usepackage{float}

\usepackage[utf8]{inputenc}
\usepackage{apptools}
\usepackage[margin=1in]{geometry}
\allowdisplaybreaks
\newcommand{\stkout}[1]{\ifmmode\text{\sout{\ensuremath{#1}}}\else\sout{#1}\fi}

\renewcommand{\fint}{\mathop{\int\hskip -0,93em -\, \!\!\!}}

\newcommand{\R}{\mathbb{R}}

\newcommand{\Rn}{\mathbb{R}^n}
\usepackage{graphicx,enumitem,dsfont,upgreek}
 \usepackage[dvips]{epsfig}
 \usepackage[mathscr]{eucal}
\usepackage{amscd}
\usepackage{amssymb,nicefrac}
\usepackage{amsthm}
\usepackage{amsmath}
\usepackage{latexsym}
\usepackage{dsfont}
\usepackage{upref}
\usepackage{hyperref}

\usepackage{color}
\theoremstyle{plain}

\newtheorem{thm}{Theorem}[section]
\theoremstyle{plain}
\newtheorem{lem}[thm]{Lemma}
\newtheorem{prop}[thm]{Proposition}

\newtheorem{defi}[thm]{Definition}

\theoremstyle{definition}
\newtheorem*{maintheorem*}{Main Theorem}
\newtheorem*{maincorollary*}{Main Corollary}
{%
\setcounter{enumi}{0}

\begin{enumerate}}%
{\end{enumerate} }

{%
\setcounter{enumi}{0}

\begin{enumerate}}%
{\end{enumerate} }

\newcommand{\norm}[1]{\ensuremath{\left\|#1\right\|}}

\newcommand{\cone}{\ensuremath{\mathcal{C}}}
\newcommand{\cD}{\ensuremath{\mathcal{D}}}
\newcommand{\cE}{\ensuremath{\mathcal{E}}}

\newcommand{\cV}{\ensuremath{\mathcal{V}}}

\newcommand{\dist}{{\rm dist}}
\newcommand{\xb}{\ensuremath{\bar{x}}}
\newcommand{\yb}{\ensuremath{\bar{y}}}
\newcommand{\abar}{\ensuremath{\bar{a}}}

\newcommand{\grad}{\nabla}

\newcommand{\Ld}{\ensuremath{L_{\Delta}}}

\newcommand{\dx}{\ensuremath{\, {\rm d}x}}
\newcommand{\dy}{\ensuremath{\, {\rm d}y}}
\newcommand{\dz}{\ensuremath{\, {\rm d}z}}

\newcommand{\dr}{\ensuremath{\, {\rm d}r}}
\newcommand{\dt}{\ensuremath{\, {\rm d}t}}

\newcommand{\supp}{\ensuremath{\mathrm{supp}\,}}

\numberwithin{equation}{section} \allowdisplaybreaks

\title[H\"older and Harnack for the logarithmic Laplacian]{H\"older regularity and Harnack inequality for the logarithmic Laplacian}

\begin{document}

\author{Anup Biswas}

\author{Subhajit Roy}

\author{Aniket Sen}

\author{Lovelesh Sharma}

\address{Indian Institute of Science Education and Research-Pune, Dr.\ Homi Bhabha Road, Pashan, Pune 411008, INDIA.}
\email{anup@iiserpune.ac.in}
\email{subhajit.roy@iiserpune.ac.in}
\email{aniket.sen@students.iiserpune.ac.in}
\email{lovelesh.sharma@iiserpune.ac.in}

\begin{abstract}
In this article, we establish Schauder-type estimates for the logarithmic Laplacian. We show that for a $\upkappa$-H\"older inhomogeneous term, the solution is also  $\upkappa$-H\"older in the interior. In fact, the interior regularity slightly exceeds $\upkappa$-H\"older smoothness up to a logarithmic correction. Additionally, we prove a Harnack inequality for non-negative solutions.
\end{abstract}


\maketitle


\section{Introduction}
In this article, we consider the logarithmic Laplacian defined by
\begin{equation}\label{E1.1}
\Ld u(x) := c_n \int_{\mathbb{R}^n} \frac{u(x) \mathbbm{1}_{B_1}(z) - u(x+z)}{|z|^n}  \dz + \rho_n u(x),
\end{equation}
where
$$
c_n:=\pi^{-\frac{n}{2}}\Gamma\left(\frac{n}{2}\right),
\qquad
\rho_n:=2\log 2 + \psi\left(\frac{n}{2}\right)-\gamma.
$$
Here, $\Gamma$ is the Gamma function, $\psi=\frac{\Gamma'}{\Gamma}$ is the Digamma function, and $\gamma = -\Gamma ' (1)$ is the Euler-Mascheroni constant. 
From \cite[Remark 1.2]{CW19}, we also have
$$
\rho_n = -2\gamma + \sum\limits_{k=1}^{\frac{n-1}{2}} \frac{2}{2k-1}\quad \text{for odd } n, \qquad \text{and} \qquad \rho_n = 2(\log 2 - \gamma) + \sum_{k=1}^{\frac{n-2}{2}} \frac{1}{k} \quad \text{for even } n.$$
Since $\gamma\approx 0.58<\log2<1$, it readily follows that $\rho_n \geq 0$ for $n\geq 2$. This operator arises naturally as the first-order expansion of the fractional Laplacian $(-\Delta)^s$ as $s\to 0^+$. More precisely, as shown in \cite[Theorem 1.1]{CW19}, for any $u\in C^\gamma_c(\Rn), \gamma>0$, we have
$$ \frac{(-\Delta)^su-u}{s}\to \Ld u \quad \text{in} \; L^p(\Rn), 1<p\leq \infty, \quad \text{as}\; s\to 0^+.$$
The logarithmic Laplacian $\Ld$ belongs to the class of weakly singular operators with Fourier symbol $2\log|\cdot|$, serving as a borderline case between the fractional Laplacian and integrable operators. Furthermore, $\Ld$ can be viewed as a close relative of the L\'evy operator associated with geometrically $\alpha$-stable processes, which have Fourier symbol $\log(1+|\cdot|^\alpha)$ for $\alpha\in(0,2]$. Several properties of this latter class of operators have been investigated in the literature (see, e.g., \cite{Bass, LB14, F23, KM17}). It is worth pointing out that $\Ld$ cannot be written as the generator of a L\'evy process, as its Fourier symbol changes sign. Consequently, standard probabilistic tools may not be directly applicable.

Given a domain $\Omega$ in $\Rn$, we study the following equation
\begin{equation}\label{E1.2}
\Ld u(x) =f(x)\quad  \text{ in }\; \Omega.
\end{equation}
The solution is understood in the weak sense. To define the notion of solution we need the following spaces
\begin{align*}
L^1_0(\Rn)=\{u:\Rn\to \R\; \text{measurable satisfying}\; \int_{\Rn}\frac{|u(z)|}{(1+|z|)^n}\dz\,<\infty\}.
\end{align*}
By $\cV(\Omega)$ we denote the collection of all functions $v\in L^1_0(\Rn)$ that satisfy
$$\iint_{x,y\in\Omega, |x-y|\leq 1}\frac{|v(x)-v(y)|^2}{|x-y|^n}\dx\dy + \int_\Omega v^2(x)\dx< +\infty.$$
The bilinear form associated to \eqref{E1.1} is given by
\begin{align*}
\cE(u, v) &=\frac{c_n}{2}\iint_{x,y\in\Rn, |x-y|\leq 1} \frac{(u(x)-u(y))(v(x)-v(y))}{|x-y|^n}\dx\dy
\\
&\quad -c_n\iint_{x,y\in\Rn, |x-y|> 1} \frac{u(x)v(y)}{|x-y|^n}\dx\dy + \rho_n\int_{\Rn} u(x)v(x)\dx.
\end{align*}
We say $u\in\cV(\Omega)$ is a weak subsolution to \eqref{E1.2} if for all $\phi\in C^\infty_c(\Omega), \phi\geq 0$, we have
$$\cE(u, \phi)\leq \int_{\Omega} f(x)\phi(x)\dx.$$
Similarly, $u$ is said to be a weak supersolution, if $-u$ is a subsolution and a weak solution if it is both sub and supersolution.

The study of the logarithmic Laplacian was initiated by Chen and Weth \cite{CW19}, who established a small domain maximum principle, Faber-Krahn inequality for $\Ld$ and  first-order asymptotics for the principal eigenvalues and eigenfunctions of $(-\Delta)^s$ as $s \to 0^+$. Since then, there has been a growing body of work on $\Ld$. For instance, Chen and V\'eron \cite{CV23} investigated the eigenvalues of $\Ld$,
Feulefack, Jarohs, and Weth \cite{FJW22} studied convergence of higher eigenvalues and eigenfunctions of $(-\Delta)^s$ as $s\to 0^+$, Jarohs and Weth \cite{JW19} proved a strong maximum principle, and Chen and V\'eron \cite{CV24} studied parabolic equations associated with $\Ld$. For the asymptotic behavior of solutions to $(-\Delta)^s$ as $s \to 0^+$, we refer the reader to Hern\'andez-Santamar\'ia and Salda\~na \cite{HSS}, Jarohs, Salda\~na, and Weth \cite{JSW20}, and Jarohs, Sen, and Weth \cite{JSW26}. Optimal boundary regularity was considered by Hern\'andez-Santamar\'ia, L\'opez-R\'ios and Salda\~na \cite{HSLS}.
 Furthermore, we mention the recent work \cite{DJS26} on the logarithmic $p$-Laplacian, obtained as the derivative of $(-\Delta_p)^s$ at $s = 0$.

The main theme of this paper centers on the following two fundamental questions:
\begin{enumerate}
    \item[(A)] Under what conditions on $f$ does $u$ become a classical solution of \eqref{E1.2} in $\Omega$?
    \item[(B)] Does the operator $\Ld$ satisfy a Harnack inequality?
\end{enumerate}

It is well known that $\Ld u$ is defined classically whenever $u$ is Dini continuous (see \cite[Proposition~1.3]{CW19}). Therefore, to address question~(A), it is natural to seek conditions on $f$ that guarantee the Dini continuity of $u$. In \cite{KM17}, Kassmann and Mimica showed that if $f \in L^\infty(\Omega)$, then the solution $u$ is continuous in $\Omega$ with a $\beta$-log-H\"older modulus of continuity for some $\beta \in (0, 1)$. Up-to-the-boundary continuity for Dirichlet problems was subsequently established by Chen and Weth \cite[Theorem~1.11]{CW19}. Furthermore, Feulefack and Jarohs \cite[Theorem~1.6]{FJ23} proved that if $f \in C^{2m}(\Omega)$, then $u \in H^m_{\rm loc}(\Omega)$; in particular, $u$ is smooth in $\Omega$ provided $f$ is smooth in $\Omega$. We also mention the work \cite{JKW25}, where local boundedness and a uniform modulus of continuity were obtained for a large class of weakly singular operators.

Since a $(1+\beta)$-log-H\"older function for $\beta > 0$ is inherently Dini continuous (and thus yields a classical solution to $\Ld u = f$), Chang-Lara and Salda\~na \cite{CLS24} investigated the existence of classical solutions with Dirichlet boundary data that are $(1+\beta)$-log-H\"older in the interior. The proofs in \cite{CLS24} rely on the Fredholm alternative, where the Dirichlet boundary data plays a key role in defining the mapping between the underlying function spaces.

In this paper, we study interior regularity regardless of boundary conditions. Our first primary result establishes that if the source term is H\"older continuous, the solution is likewise H\"older continuous.
\begin{thm}[H\"older regularity]\label{T1.1}
Let $n \geq 2$, $u \in L^\infty(\Rn) \cap \cV (\Omega)$ be a weak solution to \eqref{E1.2}.
Assume that $f \in C^{0,\upkappa}(\bar\Omega)$ for some $\upkappa \in (0,1)$. Then $u\in C^{0,\upkappa}_{\rm loc} (\Omega)$. Furthermore, for any $\Omega_1\Subset\Omega$, there exists a constant $C$ satisfying
$$ \sup_{x\neq y, x,y\in\Omega_1}\frac{|u(x)-u(y)|}{|x-y|^{\upkappa}}\leq C\left(\norm{u}_{L^\infty(\Rn)} + [f]_{C^{\upkappa}(\Omega)}\right),$$
where $C$ depends on $n, \upkappa, \Omega_1$ and $\dist(\Omega_1, \partial\Omega)$. Here $[f]_{C^{\upkappa}(\Omega)}$ denotes the 
$\upkappa$-H\"{o}lder seminorm in $\Omega$.
\end{thm}

In our next result we improve the previous regularity by a logarithmic correction.
\begin{thm}[Logarithmic improvement of H\"older regularity]\label{T1.2}
Let $n \geq 2$, $u \in L^\infty(\Omega)\cap \cV (\Omega)$ be a weak solution to \eqref{E1.2} and $f \in C^{0,\upkappa}(\bar\Omega)$ for some $\upkappa \in (0,1)$. Then there exists $\beta=\beta(n) \in (0,1)$ such that for any $\Omega_1\Subset\Omega$ we have the following estimate
$$|u(x) - u(y)| \leq C \left(\norm{u}_{L^\infty(\Omega)} + \norm{u}_{L^1_0} +[f]_{C^{\upkappa}(\Omega)}\right) \frac{|x-y|^{\upkappa}}{\left| \log|x-y| \right|^\beta }\quad  \text{ for all } x,y \in \Omega_1,$$
where the constant $C$ depends on $n, \upkappa, \Omega_1$ and $\dist(\Omega_1, \partial\Omega)$. 
\end{thm}

Our proofs of Theorems~\ref{T1.1} and~\ref{T1.2} rely on the Ishii-Lions approach introduced in \cite{IL90}. Thanks to the interior continuity result of \cite{KM17}, we show that every weak solution is also a viscosity solution (see Proposition~\ref{P2.4}). Once in the viscosity setting, we can apply the nonlocal Ishii-Lions method. Recently, the Ishii-Lions technique has been successfully employed to address the regularity of several nonlocal operators (see, e.g., \cite{BCI11,BCCI12,BS25,BT25,BT24,BQT25,CGT22}). With Theorem~\ref{T1.1} established, Theorem~\ref{T1.2} follows from \cite{KM17} with suitable modifications.

Our next main result is the Harnack inequality.
\begin{thm}[Harnack inequality]\label{T1.3}
Let $n \geq 2$, $u \in L^\infty(\Omega) \cap  \cV (\Omega)$ be a weak solution to 
$$\left\{\begin{array}{ll}
\Ld u = f & \text{in}\; \Omega,
\\[2mm]
u\geq 0 & \text{in}\; \Omega,
\end{array}
\right.$$
where $f\in L^\infty(\Omega)$. Suppose that $B_{3r}\subset \Omega$ for some $r\in (0, \frac{1}{4})$.
Then for some constant $C=C(n, r)$ we have
$$\sup_{B_{\frac{r}{8}}} u \leq C\left[\inf_{B_{\frac{r}{8}}} u + \norm{f}_{L^\infty(B_{r})} + \int_{B^{c}_{3r}} \frac{u_{-}(z)}{|z|^n}\dz \right].$$
\end{thm}
For the proof of the Harnack inequality, we do not work directly with weak solutions. This is partly because the standard method requires a Sobolev-type embedding for this class of operators, which is unavailable here. Instead, we again rely on Proposition~\ref{P2.4} to pass to viscosity solutions and then, exploiting linearity, adapt the approach of \cite{DRV22} to establish our result. We also point out that our proofs of Theorems~\ref{T1.1}--\ref{T1.3} require the zeroth-order coefficient $\rho_n$ to be non-negative, which is why we restrict our attention to $n \ge 2$.

The rest of the paper is organized as follows. In the next section, we recall several key preliminary results and introduce the notion of viscosity solutions. In Section~\ref{Holder}, we prove H\"older regularity and its logarithmic improvement, while the Harnack inequality is established in Section~\ref{Harnack}.

\section{Viscosity solution and preliminary results}
In this section, we introduce the notion of viscosity solution and provide its relation with the weak solution. We also recall a few results from the existing 
literature which will be helpful for us. From now on, we consider a slightly more general nonlocal kernel $K:\Rn\setminus\{0\}\to \R$ which is symmetric (that is, $K(y)=K(-y)$) and satisfies
\begin{equation}\label{kernel}
\kappa^{-1} |y|^{-n}\leq K(y)\leq \kappa |y|^{-n}\quad \text{for}\; y\neq 0,
\end{equation}
for some $\kappa>0$. We also assume that $K\in C^1(\Rn\setminus\{0\})$ and satisfies
\begin{equation}\label{kernel-der}
|\grad K(z)|\leq \kappa |z|^{-n-1}\quad \text{for}\; z\neq 0.
\end{equation}
The bilinear form $\cE(\cdot, \cdot)$ is also modified in an obvious fashion to $\cE_K(\cdot, \cdot)$ and the notion of 
weak sub/super solutions should also be understood with respect to $\cE_K(\cdot, \cdot)$. Furthermore, we use the notation
$L_K$ to define the operator
$$L_K u(x):= \int_{\Rn} \left(u(x) \mathbbm{1}_{B_1} (z) - u(x+z)\right) K(z) \dz.$$
Note that we do not include $\rho_n$ in the above definition of the operator. In particular, $L_K+\rho_n$ would correspond to the bilinear form
$\cE_K(\cdot, \cdot)$.

We begin with the following result of Kassmann and Mimica \cite[Theorem 3]{KM17}.
\begin{thm}\label{T-KM}
Let $\widehat{K}:\Rn\setminus\{0\} \to[0, \infty)$ be a symmetric, measurable function satisfying
$$\kappa^{-1}\mathbbm{1}_{B_1}(y)|y|^{-n}\leq \widehat{K}(y)\leq \kappa \mathbbm{1}_{B_1}(y)|y|^{-n}$$
for some $\kappa>0$. Then there exists a $\beta=\beta(n,\kappa)\in (0, 1)$ and $C=C(n, \kappa, r)>0$, so that for any
$u\in C^2_b(\Rn)$ solving
$$\int_{\Rn} (u(x)-u(x+z))\widehat{K}(z)\dz=f\quad \text{in} \; B_r$$
for $r\leq \frac{1}{2}$ and some $f\in L^\infty(B_r)$, we have
\begin{equation*}\label{KM log}
|u(x)-u(y)|\ \leq C \left(\norm{u}_\infty + \norm{f}_{L^\infty(B_r)}\right) \left |\log|x-y| \right|^{-\beta}\quad \text{for all}\; x, y\in B_{\frac{r}{4}}.
\end{equation*}
\end{thm}

Applying Theorem~\ref{T-KM} and a routine approximation argument we obtain the following interior regularity.
\begin{thm}\label{T2.2}
Let $u\in L^\infty(\Rn) \cap \cV(B_{2r}) $ be a weak solution to $L_K u + \rho(x) u(x) = f$ in $B_{2r}$, where $\rho$ and $f$ are bounded measurable functions. Then $u \in C(B_{2r})$ and there exist constants  $C=C(n,\kappa, r, \norm{\rho}_\infty)>0$ and $\beta=\beta(n,\kappa) \in (0,1)$ 
such that  for all $x,y \in B_{r/4}$ we have
$$|u(x) - u(y)| \leq C\left(\norm{u}_\infty + \norm{f}_{L^\infty(B_{2r})} + \int_{\Rn} \frac{|u(z)|}{(1+|z|)^n}\dz\right) |\log|x-y||^{-\beta}.$$
\end{thm}

\begin{proof}
Since $u\in L^\infty(\Rn)$, moving $\rho u$ to the right hand side, we may assume that $\rho=0$. Moreover, it is enough to complete the proof assuming
$r\leq \frac{1}{4}$. For a larger $r$ we can employ a standard covering argument to conclude the result.

Define $\varphi(x) = \phi(x+z) \eta_\varepsilon(z)$, where $\phi \in C_c^\infty(B_{r})$ and $\eta_\varepsilon\geq 0$ is the standard (scaled) mollifier. We denote $\eta_\varepsilon$ by $\eta$. By the definition of weak solution
\begin{align*}
&\frac{1}{2} \iint_{\stackrel{x,y \in \Rn}{\text{\tiny $|x-y|\leq 1$}}} (u(x)-u(y))(\varphi(x)-\varphi(y)) K(x-y)\dx \dy 
\\
&\qquad
-\iint_{\stackrel{x,y \in \Rn}{\text{\tiny $|x-y|> 1$}}} u(x)\varphi(y) K(x-y) \dx \dy = \int_{\Rn} f(x) \varphi(x) \dx.
\end{align*}
Applying change of variables $x\mapsto x-z, y \mapsto y-z$ we have
\begin{align*}
&\frac{1}{2} \iint_{\stackrel{x,y \in \Rn}{\text{\tiny $|x-y|\leq 1$}}} (u(x-z)-u(y-z))\eta(z)(\phi(x)-\phi(y))K(x-y)\dx \dy 
\\
&\qquad- \iint_{\stackrel{x,y \in \Rn}{\text{\tiny $|x-y|> 1$}}} u(x-z)\eta(z)\phi(y) K(x-y)\dx \dy = \int_{\Rn} f(x-z)\eta(z) \phi(x) \dx.
\end{align*}
Now, integrating over $\Rn$ with respect to $z$ and applying the dominated convergence theorem we get
\begin{align*}
&\frac{1}{2} \iint_{\stackrel{x,y \in \Rn}{\text{\tiny $|x-y|\leq 1$}}} (u^\varepsilon(x)-u^\varepsilon(y))(\phi(x)-\phi(y)) K(x-y) \dx \dy 
\\
&\qquad - \iint_{\stackrel{x,y \in \Rn}{\text{\tiny $|x-y|> 1$}}} u^\varepsilon(x)\phi(y) K(x-y) \dx \dy = \int_{\Rn} f^\varepsilon(x) \phi(x) \dx,
\end{align*}
where $u^\varepsilon=u\ast\eta^\varepsilon$ and $f^\varepsilon=f\ast\eta^\varepsilon$.
Note that $u^\varepsilon\in C^2_b(\Rn)$ is a solution to
 $$L_{\widehat{K}} u^\varepsilon =f^\varepsilon+\int_{B_1^c}u^\varepsilon(x+z)K(z)\dz,$$
 where $\widehat{K}(y)=K(y)\mathbbm{1}_{B_1}(y)$.
Since  $u\in  L_0^1(\Rn)$ the right side is a bounded function, and hence using  Theorem~\ref{T-KM} and then letting $\varepsilon \rightarrow 0$ we conclude the result.
\end{proof}

Let us now introduce the notion of viscosity solution which will play the central role in our analysis.
\begin{defi}[Viscosity Solution]\label{Defi-vis}
Let $\rho, f : \Omega \to \mathbb{R}$ be continuous functions. A function $u : \Rn \to \mathbb{R}$ is a viscosity subsolution to 
$$
L_K u + \rho u = f \quad \text{in } \Omega
$$
if $u \in L^1_0(\mathbb{R}^n)\cap USC(\Omega)$ and, for every test function $\varphi \in C^{0,\alpha}(\Omega)$ with $\alpha \in (0,1)$ and every point $x_0 \in \Omega$ such that $u(x_0) = \varphi(x_0)$ and $\varphi \ge u$ in some ball $B_\delta(x_0) \subset \Omega$, we have
$$
L_K \varphi_\delta(x_0) + \rho(x_0) \varphi_\delta(x_0) \le f(x_0),
$$
where $\varphi_\delta : \mathbb{R}^n \to \mathbb{R}$ is defined by
$$
\varphi_\delta(y) = 
\begin{cases} 
\varphi(y) & \text{for } y \in B_\delta(x_0), \\[2mm] 
u(y) & \text{for } y \notin B_\delta(x_0). 
\end{cases}
$$
A function $u : \R^n \to \mathbb{R}$ is a viscosity supersolution if $-u$ is a viscosity subsolution. Finally, $u$ is a viscosity solution if it is both a viscosity subsolution and a viscosity supersolution in $\Omega$. 
\end{defi}
We conclude this section by establishing a link between weak and viscosity solutions, enabling the application of viscosity solution results to weak solutions.
\begin{prop}\label{P2.4}
Let $u \in \cV(\Omega) \cap USC(\Omega)$ be a weak subsolution of \eqref{E1.2}, and $f\in C(\Omega)$ . Then $u$ is a viscosity subsolution to \eqref{E1.2}. Similar conclusion holds for weak supersolutions and solutions.
\end{prop}

\begin{proof}
Let $x_0 \in \Omega$ and $\varphi \in C^{0,\alpha}(\Omega)$, $\alpha \in (0,1)$, be such that $u(x_0) = \varphi(x_0)$ and $\varphi \ge u$ in some ball $B_\delta(x_0) \subset \Omega$. Due to the monotonicity property of the integral, we may restrict $\delta$ small enough so that
$|B_{\delta}(x_0)| < 2^n e^{\frac{n}{2}(\psi(\frac{n}{2})-\gamma)}|B_1|$. This restriction will let us use the maximum principle in \cite[Corollary~1.9]{CW19}. Define
\begin{align*}
\varphi_\delta =
\begin{cases}
\varphi(x) \text{ if } x \in B_\delta(x_0),
\\[2mm]
u(x) \text{ otherwise.}
\end{cases} 
\end{align*}
We claim that $\Ld \varphi_\delta(x_0) \leq f(x_0)$. If not, then there exists $\epsilon>0$ such that $\Ld \varphi_\delta(x_0) > f(x_0) +2 \epsilon$.
Since $\varphi_\delta$ is H\"older continuous in $B_\delta(x_0)$, it is uniformly Dini continuous, and therefore, by   \cite[Proposition 1.3]{CW19}, $\Ld \varphi_\delta$ is continuous in $B_\delta(x_0)$. By continuity, there exists $\delta_1 < \delta$ satisfying
\begin{equation}\label{P2.4A}
\Ld \varphi_\delta(x) > f(x) +\epsilon\quad \text{ for } x \in B_{\delta_1}(x_0).
\end{equation}
Let $\eta \in C_c^\infty (B_{\frac{\delta_1}{2}}(x_0))$ be such that $\eta \equiv 1$ in $B_{\frac{\delta_1}{4}}(x_0)$. For $\theta>0$, define $v_\theta =\varphi_\delta - \theta \eta$. Again, by \cite[Proposition 1.3]{CW19}, $\Ld \eta$ is continuous in $B_\delta(x_0)$. Choose $\theta\in (0,1)$ small such that 
\begin{equation}\label{P2.4B}
\theta |\Ld \eta(x)| < \frac{\epsilon}{4}\quad  \text{ for } x\in B_{\delta_1}(x_0).
\end{equation}
Hence, for $x\in B_{\delta_1}(x_0)$, we have
\begin{align*}
\Ld v_\theta(x)& = \Ld \varphi_\delta(x) -\theta \Ld \eta(x)  
\\
&> f(x) +\frac{3\epsilon}{4},
\end{align*}
using \eqref{P2.4A} and \eqref{P2.4B}. Since $v_\theta$ is a classical solution, it is easily seen that
$\Ld (v_\theta - u) \geq \frac{3\epsilon}{4}$ in $B_{\delta_1}(x_0)$, in the weak sense. Again, $\varphi_\delta \geq u$ and $\eta \equiv 0$ in $B^c_{\delta_1}(x_0)$, giving us $v_\theta \geq u$ in $B^c_{\delta_1}(x_0)$. By  \cite[Corollary 1.9 (ii)]{CW19} and our choice of $\delta$, 
we then obtain $v_\theta\geq u$ in $\Rn$. Since $\varphi_\delta (x_0) = u(x_0)$ and $\eta(x_0)=1$, we get
$$v_\theta(x_0) \geq u(x_0) \Rightarrow -\theta \geq 0,$$
which contradicts the positivity of $\theta$. Therefore, our claim holds and $u$ is a viscosity subsolution. Hence the proof.
\end{proof}

\section{H\"older regularity and its logarithmic improvement}\label{Holder}
In this section we prove H\"{o}lder regularity of the viscosity solution. Throughout this section we consider the problem
\begin{equation}\label{E3.1}
L_K u + \rho(x) u= f,
\end{equation}
where the kernel $K$ satisfies \eqref{kernel} and \eqref{kernel-der}. The main result of this section is as follows.
\begin{thm}\label{T3.1}
Let $0 < r_1 < r_2<1$ and let $u \in C(\bar{B}_{r_2}) \cap L^\infty(\mathbb{R}^n) \cap L^1_0(\mathbb{R}^n)$ be a viscosity solution to \eqref{E3.1} in $B_{r_2}$. Suppose that $\rho, f \in C^{0, \upkappa}(\bar{B}_{r_2})$ with $\rho \geq 0$ for some 
$\upkappa \in (0, 1)$. 
Define $\upkappa_0 = 0$ and $\upkappa_m = \frac{1}{2-\upkappa_{m-1}}$ for $m\geq 1$. If $u \in C^{0,\upkappa_i} (\bar{B}_{r_2})$, then $u \in C^{0,\min \{\upkappa, \upkappa_{i+1}\} } (\bar{B}_{r_1})$. Moreover, the $C^{0,\min \{\upkappa, \upkappa_{i+1}\}}$-H\"older seminorm of $u$ in $B_{r_1}$ depends only on $n,\kappa$, $\upkappa, \upkappa_i, \upkappa_{i+1}$, $r_1$, $r_2, \norm{u}_\infty$, $\|u\|_{C^{0, \upkappa_i}(B_{r_2})}$,
$[\rho]_{C^{0, \upkappa}(B_{r_2})}$, and $[f]_{C^{0, \upkappa}(B_{r_2})}$.

In particular, if $u \in C(\Omega) \cap L^\infty(\mathbb{R}^n) \cap L^1_0(\mathbb{R}^n)$ is a viscosity solution to \eqref{E3.1} in $\Omega$, and $\rho, f \in C^{0, \kappa}(\bar{\Omega})$ with $\rho \geq 0$, then $u \in C^{0, \upkappa}_{\mathrm{loc}}(\Omega)$. Furthermore, for any subdomain $\Omega_1 \Subset \Omega$,
$$
\sup_{\substack{x, y \in \Omega_1 \\ x \neq y}} \frac{|u(x) - u(y)|}{|x - y|^{\upkappa}} \leq C \left( \|u\|_{L^\infty(\mathbb{R}^n)} + [f]_{C^{0, \kappa}(\Omega)} \right),
$$
where $C > 0$ depends on $n$, $\kappa,\upkappa, \Omega_1$, $\norm{\rho}_{C^{0, \upkappa}(\Omega)}$ and $\mathrm{dist}(\Omega_1, \partial\Omega)$.
\end{thm}

\begin{proof}
Dividing both sides of \eqref{E3.1} by $(\norm{u}_\infty+\|u\|_{C^{0, \upkappa_i}(B_{r_2})} + [f]_{C^{0,\upkappa}(B_{r_2})})$ we assume that
$$\norm{u}_\infty\leq 1,\quad \|u\|_{C^{0, \upkappa_i}(B_{r_2})}\leq 1\quad \text{and}\quad [f]_{C^{0,\upkappa}(B_{r_2})}\leq 1.$$
Denote by $\upgamma=\min\{\upkappa, \upkappa_{i+1}\}$. We assume $\upgamma>\upkappa_i$, otherwise, there is nothing to prove.
To apply the Ishii-Lions method, we consider the doubling variable function
$$
\Phi (x,y) = u(x)-u(y) -L |x-y|^{\upgamma} - 2 \psi(x),
$$
where $\psi:\Rn\to[0, 1]$ is a smooth 
function satisfying  $\psi =0$ in  $B_{r_1}$ and $\psi =1$ in  $B^c_{\frac{r_1+r_2}{2}}$. To complete the proof we need to find an $L$, independent of $u$,
so that $\Phi\leq 0$ in $B_{r_2}\times B_{r_2}$. Once this is established the result follows from the fact that
$$\Phi(x, y)\leq 0\quad \text{in}\; \bar{B}_{r_1}\times \bar{B}_{r_1}\Rightarrow |u(x)-u(y)|\leq L|x-y|^{\upgamma}\quad \text{for}\; x, y\in \bar{B}_{r_1}.$$
To prove the above claim we proceed via the method of contradiction and assume that for all large $L$ we have
\begin{equation}\label{E3.1A}
\max_{\bar{B}_{r_2}\times \bar{B}_{r_2}}\Phi>0.
\end{equation} 
Let $(\xb,\yb)\in \bar{B}_{r_2}\times \bar{B}_{r_2}$ be a point of maximizer, that is, 
$$\max_{\bar{B}_{r_2}\times \bar{B}_{r_2}}\Phi=\Phi(\xb,\yb)>0.$$
Note that, by our choice of $\psi$,  we have  $\Phi(x,y) \leq 0$ for $|x| \geq \frac{r_1+r_2}{2}$. Thus $|\xb|< \frac{r_1+r_2}{2}$.
Again, since
\begin{equation}\label{E3.1B}
L  |\xb -\yb|^{\upgamma} \leq u(\xb)-u(\yb)\leq 2,
\end{equation}
setting $L > 2 \left[ \frac{r_2 - r_1}{8} \right]^{-\upgamma}$ and $\abar:=\xb-\yb$,
we see that $|\bar a| = |\xb-\yb| < \frac{r_2 - r_1}{8}$. In particular,  $|\yb|< \frac{3r_1+ 5r_2}{8}$.

Denote by 
$$\phi(x,y) =L |x-y|^{\upgamma} + 2\psi(x) .$$
Then $x \mapsto u(\xb) -\phi(\xb, \yb) +\phi(x, \yb)$ touches $u$ from above at $\xb$, and $y \mapsto u(\yb) +\phi(\xb, \yb) -\phi(\xb, y)$ touches $u$ from below at $\yb$ in any ball of radius smaller that $\frac{r_2-r_1}{4}$.  For $\delta \in (0, \frac{r_2-r_1}{4})$, we let
$$w_1(z) = 
\begin{cases}
u(\xb) -\phi(\xb, \yb) +\phi(z, \yb) &\text{ for } z \in B_{\delta}(\xb),
\\[2mm]
u(z) &\text{otherwise},
\end{cases}$$
and 
$$w_2(z) = 
\begin{cases}
u(\yb) +\phi(\xb, \yb) -\phi(\xb, z) &\text{ for } z \in B_{\delta}(\yb),
\\[2mm]
u(z) &\text{otherwise}.
\end{cases}$$
Thus, by the definition of viscosity solutions (Definition~\ref{Defi-vis}) we get
$L_K w_1 (\xb) +\rho(\xb) u(\xb) \leq f(\xb)$ and $L_K w_2 (\yb) +\rho(\yb)u(\yb)\geq f(\yb)$, implying
\begin{align}\label{E3.1C}
L_K w_1(\xb) -L_K w_2 (\yb) &\leq f(\xb) - f(\yb) - \rho(\xb)(u(\xb)-u(\yb))+ u(\yb)(\rho(\yb)-\rho(\xb)) \nonumber
\\
&\leq |\abar|^\upkappa
+ [\rho]_{C^{0, \upkappa}(B_{r_2})}|\abar|^\upkappa,
\end{align}
where we used the fact $u(\xb)-u(\yb)>0$ and $\rho\geq 0$. We introduce the notation
$$L_K [D] w(x) = \int_{D} (w(x) \mathbbm{1}_{B_1} (z) - w(x+z))K(z) \dz,$$
and define the following domains (see Figure \ref{fig})
$$
\cone=\{z\in B_{\delta_0|\abar|}\; :\; |\langle \abar, z\rangle|\geq (1-\eta_0) |\abar||z|\}, \quad \cD_1=B_{\delta_1 |\abar|} \cap \cone^c,
$$ 
$$
\quad \cD_2=B_{|\abar|^{\upgamma}}\setminus (\cD_1\cup\cone)
\quad \cD_3 = B_{\tilde r} \setminus B_{|\abar|^{\upgamma}}, \text{ and}\quad \cD_4 = B^c_{\tilde r}.
$$
where $\tilde{r} = \frac{r_2 - r_1}{4}$ and $\delta_0, \eta_0, \delta_1$ are small, to be chosen later. We also set $\delta=\delta_1|\abar|$.
Note that, since $r_2<1$, we have $\tilde{r} <1$.
\begin{figure}[h]
    \centering
    \begin{tikzpicture}[scale=0.58]

    \definecolor{colorE}{rgb}{1,1,1}      
    \definecolor{colorD}{HTML}{F2F2F2}     
    \definecolor{colorA}{HTML}{E6F0FA}      
    \definecolor{colorC}{HTML}{FFEBE6}       
    \definecolor{colorB}{HTML}{E6F5EA}        

\fill[colorD] (0,0) circle (4.5);

\fill[colorA] (0,0) circle (3.2);
    
\fill[colorC] (0,0) -- (50:2.5) arc (50:80:2.5) -- cycle;
\fill[colorC] (0,0) -- (230:2.5) arc (230:260:2.5) -- cycle;
    
\fill[colorB] (0,0) -- (80:0.7) arc (80:230:0.7) -- cycle;
\fill[colorB] (0,0) -- (260:0.7) arc (260:410:0.7) -- cycle;

\draw[thick] (0,0) circle (4.5);
\draw[thick, dashed] (0,0) circle (3.2);
\draw[thick] (50:2.5) -- (230:2.5);
\draw[thick] (80:2.5) -- (260:2.5);
\draw[thick] (50:2.5) arc (50:80:2.5);
\draw[thick] (230:2.5) arc (230:260:2.5);
\draw[thick] (80:0.7) arc (80:230:0.7);
\draw[thick] (260:0.7) arc (260:410:0.7);

\draw[->, thick, dashed] (0,0) -- (65:4) node[below right=-2pt] {$\bar{a}$};

\node at (335:0.36) {{\scriptsize $\mathcal{D}_1$}};
\node at (245:1.5) {$\mathcal{C}$};
\node at (315:1.8) {$\mathcal{D}_2$};
\node at (345:3.8) {$\mathcal{D}_3$};
\node at (0:5.2)   {$\mathcal{D}_4$};

\end{tikzpicture} 
    \caption{Relevant domains for nonlocal viscosity evaluation}
    \label{fig}
\end{figure}
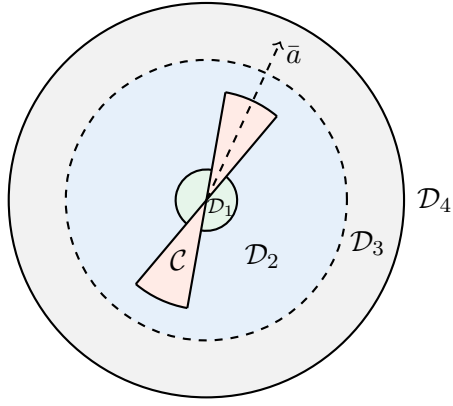
Also, with the help of \eqref{E3.1B} we set $L$ large enough, depending on $\tilde{r}$, so that $|\abar|^\upgamma< \tilde{r}$. In the calculations below,
we use the notations $C_1, C_2, ...$ to denote arbitrary constants whose value might change from line to line.
From \eqref{E3.1C} we then have
\begin{align}\label{E3.1D}
C |\abar|^\upkappa &\geq
\underbrace{L_K [\cone] w_1(\xb) - L_K [\cone] w_2(\yb)  }_{I_1}  
+ \underbrace{L_K [\cD_1] w_1(\xb) - L_K [\cD_1] w_2(\yb)}_{I_2} \nonumber
\\
&\quad
+ \underbrace{L_K [\cD_2] w_1(\xb) - L_K [\cD_2]w_2(\yb)}_{I_3}
+\underbrace{L_K [\cD_3] w_1(\xb) - L_K [\cD_3]w_2(\yb)}_{I_4} \nonumber
\\
&\qquad
+\underbrace{L_K [\cD_4] w_1(\xb) - L_K [\cD_4]w_2(\yb)}_{I_5}.
\end{align}
We denote $\triangle^2 f(x,z) =f(x) - f(x+z) + \nabla f(x)\cdot z$. We need the following estimates from \cite{BT25}. There exist 
$\delta_0=\eta_0\in (0, \frac{1}{2})$ and $L_0>0$, so that
\begin{align}
\frac{1}{C_1} L |\abar|^{\upgamma-2}|z|^2 &\leq \triangle^2\phi (\cdot, \yb)(\xb, z), \triangle^2\phi(\xb, \cdot)(\yb, z)  \leq C_1 L 
|\abar|^{\upgamma-2}|z|^2\quad \text{for all} \; z\in \cone, \label{E3.1E}
\\
-C_1 L |\abar|^{\upgamma-2}|z|^2 &\leq \triangle^2\phi (\cdot, \yb)(\xb, z), \triangle^2\phi(\xb, \cdot)(\yb, z)\leq C_1 L 
|\abar|^{\upgamma-2}|z|^2
\quad \text{for all}\; |z|\leq \delta_1 |\abar|,\label{E3.1F}
\end{align}
for all $\delta_1\in (0, \frac{1}{2})$ and $L\geq L_0$.
A proof of \eqref{E3.1E} can be found in \cite[Lemma~2.2]{BT25} and see the proof of \cite[Lemma~3.2, eq (3.2)]{BT25} for the estimate \eqref{E3.1F}.
Now we proceed to estimate $I_i, i=1,2,..,5,$ suitably so that \eqref{E3.1D} leads to a contradiction for large values of $L$.

\medskip
\noindent\underline{Lower bound of $I_1$:} We denote by $\ell(z)=-\grad_x\phi(\xb,\yb)\cdot z$. By the anti-symmetry
and linearity of $\ell$, it follows that $L_K [D]\ell(x)=0$ for all $x$, provided $D$ is symmetric about $0$. Using $u(\xb)-u(\xb+z)\geq \phi(\xb,\yb)
-\phi(\xb+z, \yb)$, we compute for $L\geq L_0$ that
\begin{align*}
L_K [\cone] w_1 (\bar x) &= L_K [\cone] w_1 (\bar x) - L_K [\cone] \ell (\bar x)
\\
&\geq L_K [\cone] \phi(\cdot , \bar y)(\bar x) - L_K [\cone] \ell (\bar x)
\\
&= \int_{\cone} \triangle^2\phi (\cdot, \yb)(\xb, z) K(z)\dz
\\
&\geq \frac{1}{C_1} L |\abar|^{\upgamma-2} \int_{\cone} |z|^{2-n} \dz
\\
&\geq C_2 L |\abar|^{\upgamma}
\end{align*}
where  in the fourth line we use \eqref{E3.1E} and in the last line we compute the integral using \cite[Example 1]{BCCI12} .
An upper bound for $L_K[\cone] w_2(\bar y)$ is obtained similarly, giving us
\begin{equation}\label{estim I1}
I_1 \geq C_1 L |\abar|^{\upgamma}
\end{equation}
for all $L\geq L_0$, where $L_0$ depends on $\upgamma, n$.

\medskip
\noindent\underline{Lower bound of $I_2$:} By a similar calculations as above, and using \eqref{E3.1F} it is easily seen that 
\begin{equation}\label{estim I2}
I_2 \geq - C_2 L|\abar|^{\upgamma} \delta^2_1
\end{equation}
for all $L\geq L_0$.
See also \cite[Lemma~3.2]{BT25} for a similar calculation.

\medskip
\noindent\underline{Lower bound of $I_3$:}
Note that, for $|z|\leq \tilde{r}$, we have $\xb+z, \yb+z\in B_{r_2}$, and therefore
\[ \Phi ( \xb +z, \yb+z) \leq \Phi (\xb, \yb) \]
\[\Rightarrow (u(\xb +z) - u(\xb)) - (u(\yb +z) - u(\yb)) \leq 2(\psi (\xb +z) - \psi(\xb)).\]
Thus
\begin{align}\label{estim I3}
I_3 &\geq  2 \int_{\cD_2} ( \psi (\xb)-\psi(\xb +z)) K(z)\dz \nonumber
\\
&\geq -2\norm{\grad\psi}_\infty \kappa \int_{\cD_2} |z|^{1-n} \dz \nonumber
\\
&\geq -2\norm{\grad\psi}_\infty \kappa \int_{0\leq r \leq |\abar|^{\upgamma}} \dr \nonumber
\\
&\geq -C_3 |\abar|^{\upgamma },
\end{align}
where we used \eqref{kernel} in the second line.

\medskip

\noindent\underline{Lower bound of $I_4$:}
\begin{align*}
I_4 &= \int \limits_{B_{\tilde r} \setminus B_{|\abar|^{\upgamma}}} \left[ u(\xb) - u(\xb +z)  - u(\yb) + u(\yb +z) \right] K(z)\dz
\\
&= \int \limits_{B_{\tilde r}(\xb) \setminus B_{|\abar|^{\upgamma}}(\xb)}  [u(\xb) - u(z)]   K(\xb-z)\dz - \int \limits_{B_{\tilde r}(\yb) \setminus B_{|\abar|^{\upgamma}}(\yb)}  [u(\yb) - u(z)]   K(\yb -z)\dz
\\
&\geq \underbrace{\int \limits_{B_{\tilde r(\xb)} \setminus B_{|\abar|^{\upgamma}}(\xb)}  [u(\xb) - u(z)]   K(\xb-z)\dz - \int \limits_{B_{\tilde r(\yb)} \setminus B_{|\abar|^{\upgamma}}(\yb)}  [u(\xb) - u(z)]  K(\xb-z)\dz }_{I_{4,1}}
\\
&\quad + \underbrace { \int \limits_{B_{\tilde r}(\yb) \setminus B_{|\abar|^{\upgamma}}(\yb)}  [u(\yb) - u(z)]   \left[ K(\xb-z) - K(\yb-z) \right] \dz}_{I_{4,2}},
\end{align*}
where in the last line we use the fact that $u(\xb)>u(\yb)$.
Since $|\abar|\to 0$ as $L$ enlarges, by \eqref{E3.1B}, and $\upgamma<1$, we have $|\abar|<\frac{1}{4}|\abar|^\upgamma$ for all
large $L$. Thus, for any $t\in [0,1]$ and $z\in B_{\tilde r}(\yb) \setminus B_{|\abar|^{\upgamma}}(\yb)$, we have
$$
|\yb-z + t\abar|\geq |\yb-z|-\frac{1}{4}|\abar|^\upgamma>\frac{1}{2}|\yb-z| + \frac{1}{4}|\abar|^\upgamma>\frac{1}{2}|\yb-z|.
$$
Applying the mean value theorem together with the $\upkappa_i$-H\"older continuity of $u$ and \eqref{kernel-der}, we get
$$
I_{4,2} \geq -C_1|\abar| \int\limits_{B_{\tilde r(\yb)} \setminus B_{|\abar|^{\upgamma}}(\yb)} \frac{|\yb -z|^{\upkappa_i}}{|\yb-z|^{n+1}}\dz \geq -C_2|\abar|^{1-\upgamma(1-\upkappa_i)}.
$$
We note that for $z\in B_{\tilde r(\xb)} \setminus B_{|\abar|^{\upgamma}}(\xb)$, we have
$$\frac{1}{2}|\xb-z|\leq |\xb-z|-|\abar|\leq |\yb-z|\leq |\xb-z|+ |\abar|\leq 2 |\xb-z|.$$
Using 
$$\underbrace{ \left(B_{\tilde r}(\xb) \setminus B_{|\abar|^{\upgamma}}(\xb)\right) \triangle \left(B_{\tilde r}(\yb) \setminus B_{|\abar|^{\upgamma}}(\yb)\right)}_{:=\mathcal{B}}  \subset \underbrace{\left( B_{|\abar|^{\upgamma} + |\abar|}(\xb) \setminus (B_{|\abar|^{\upgamma} - |\abar|}(\xb) \right)}_{:=\mathcal{B}_1} \cup \underbrace{\left( B_{\tilde r + |\abar|}(\xb) \setminus (B_{\tilde r - |a|}(\xb) \right)}_{:=\mathcal{B}_2}, $$
and $\upkappa_i$-H\"older property of $u$, we estimate
\begin{align*}
I_{4,1} &\geq\int_{{\mathcal{B}}} ( u(\xb) - u(z))K(\xb-z) \dz
\\
&\geq -C_1 \left[ \int_{\mathcal{B}_1} \frac{|\xb - z|^{\upkappa_i}}{|\xb - z|^n} \dz + \int_{\mathcal{B}_2} \frac{\dz}{|\xb - z|^n}\right]
\\
&\geq -C_1 \left[  |\abar|^{\upgamma \upkappa_{i}}\left| \log \left( \frac{|\abar|^{\upgamma} + |a|}{|\abar|^{\upgamma} - |\abar|}  \right) \right|   +   \left| \log \left( \frac{\tilde r + |\abar|}{\tilde r - |\abar|}  \right) \right|\right]
\\
&\geq -C_1 \left[  |\abar|^{\upgamma \upkappa_{i}} \left| \log \left( 1 + \frac{2|\abar|^{1- \upgamma} }{1-|\abar|^{1- \upgamma}}  \right) \right|   +   \left| \log \left(1+  \frac{2|\abar|}{\tilde r - |\abar|}  \right) \right|\right]
\\
&\geq -C_1 [|\abar|^{\upgamma \upkappa_{i}} |\abar|^{1 - \upgamma} +|\abar|]
\\
&\geq -C_1|\abar|^{1- \upgamma(1-\upkappa_i )},
\end{align*}
where in the second line we used \eqref{kernel}, in the third line we use $|\xb-z|\leq 2|\abar|^\upgamma$ in $\mathcal{B}_1$, and in the fifth
line we use the inequality $\log(1+x)\leq x$ for $x\geq 0$.
Combining the above two estimates, we arrive at
\begin{equation}\label{estim I4}
I_4 \geq -C_4 |\abar|^{1- \upgamma(1-\upkappa_i )}
\end{equation}
for all $L\geq L_0$.

\medskip

\noindent\underline{Lower bound for $I_5$:} Notice that $B^c_{\tilde r}(\xb) \setminus B^c_{\tilde r}(\yb) \subset B_{\tilde r + |\abar|}(\xb) \setminus 
B_{\tilde r - |\abar|}(\xb) $. Hence, using $u(\xb)>u(\yb)$, we have
\begin{align}\label{estim I5}
I_5 &= \int_{B^c_{\tilde r}} (u(\xb) \mathbbm{1}_{B_1} (z) - u(\xb+z)) K(z)\dz -
\int_{B^c_{\tilde r}} (u(\yb) \mathbbm{1}_{B_1} (z) - u(\yb+z)) K(z) \dz \nonumber
\\
&\geq \int_{B^c_{\tilde r}}  - u(\xb+z) K(z) \dz + \int_{B^c_{\tilde r}}  u(\yb+z) K(z) \dz \nonumber
\\
&= -\int_{B^c_{\tilde r}(\xb)} u(z) K(\xb-z) \dz +\int_{B^c_{\tilde r}(\yb) } u(z) K(\yb -z)\dz \nonumber
\\
&\geq -\int_{\left( B_{\tilde r + |\abar|}(\xb) \setminus (B_{\tilde r - |\abar|}(\xb) \right)} \frac{\dz}{|\xb - z|^n} + 
\int_{B^c_{\tilde r}(\yb) } |K(\xb-z)-K(\yb-z))|\dz \nonumber
\\
&\geq -C_5|\abar|,
\end{align}
where the last inequality follows from a similar calculations done for $I_{4,1}$ and $I_{4,2}$.

Now gathering the estimates from \eqref{estim I1}, \eqref{estim I2}, \eqref{estim I3}, \eqref{estim I4} and \eqref{estim I5} in \eqref{E3.1D} we obtain
\begin{align*}
|\abar|^\upkappa \geq C [L |\abar|^{\upgamma} (1-\delta^2)  -|\abar|^{\upgamma} - |\abar|^{1- \upgamma(1-\upkappa_i )} - |\abar|]
\end{align*}
for all $L\geq L_0$, where $L_0$ can be set as the maximum of all $L_0$ appearing in \eqref{estim I1}-\eqref{estim I5}. By our choice of $\upkappa_i$ we have $1- \upkappa_{i+1}(2-\upkappa_i ) = 0\Rightarrow \upgamma\leq 1- \upgamma(1-\upkappa_i )$, and $\upkappa\geq\upgamma$, there exists a $L_0$ so that the above inequality fail to hold for any $L\geq L_0$. Thus any such $L$ will lead to a contradiction to \eqref{E3.1A}. This completes the proof of the first part.

Since $[0, 1)\ni t\mapsto \frac{1}{2-t}-t$ is a strictly decreasing map, attaining value $0$ at $t=1$, we see that $\{\upkappa_n\}$ forms a strictly increasing sequence with its limit being $1$. Therefore, given any $\upkappa\in (0, 1)$, there exists a finite $m$ satisfying $\upkappa_m>\upkappa$.
Now we can follow a standard iteration and covering argument to see that $u\in C^{0, \upkappa}_{\rm loc}(\Omega)$. Furthermore, as suggested by the proof above, local $\upkappa_1\wedge\upkappa$-H\"older seminorm depends of $\|u\|_\infty$ and $[f]_{C^{0, \upkappa}}$, and therefore, all the successive H\"{o}lder seminorms would depend on $\|u\|_\infty$ and $[f]_{C^{0, \upkappa}}$. This completes the proof of the second part.
\end{proof}

\begin{proof}[Proof of Theorem~\ref{T1.1}]
As we mentioned before, for $n\geq 2$ we have $\rho_n>0$. 
Also, by Theorem ~\ref{T2.2} we have $u\in C(\Omega)$.
Hence the proof follows by combining Proposition~\ref{P2.4} and Theorem~\ref{T3.1}.
\end{proof}

Next we proceed to prove a logarithmic improvement on the regularity result established in Theorem~\ref{T3.1}, which will give us the proof of
Theorem~\ref{T1.2}. We first need the following lemma for it.
\begin{lem}\label{L3.2}
Consider $r\in (0,1)$.
Let $u\in L^\infty(\Rn) \cap L_0^1(\Rn) \cap C^{0,\upkappa}(\bar{B}_r)$ be a viscosity solution to $L_K u +\rho(x) u = f$ in $B_r$,
 where $f, \rho$ are in $C^{0, \upkappa}(\bar{B}_r)$. For $|e|=1$, define 
 $$w(x) = \frac{u(x+eh) -u(x)}{|h|^{\upkappa}}\quad \text{for}\; 0<|h|<\frac{r}{8}.$$
  Then there exist constants 
  $C=C(n,r,\kappa,\rho)$ and $\beta=\beta(n,\kappa) \in (0,1)$ such that  for all $x,y \in B_{\frac{r}{32}}$
$$|w(x) - w(y)| \leq C\left(\norm{u}_\infty + \norm{u}_{C^{0, \upkappa}(B_r)} + [f]_{C^{0, \upkappa}(B_r)}\right) |\log|x-y||^{-\beta}.$$
\end{lem}

\begin{proof} 
Since $u\in C^{0,\upkappa}(\bar{B}_r) \cap L^{\infty} (\Rn)\cap L_0^1(\Rn) $, we see that $u\in \cV(B_r)$. Furthermore, 
$u$ becomes a classical solution and therefore, a weak solution to 
\begin{equation}\label{E3.2A}
L_K u +\rho(x) u = f \quad \text{in}\; B_r.
\end{equation}
Also, $w \in L^{\infty} (\bar{B}_{\frac{3r}{4}})$. Since $\rho u\in C^{0,\upkappa}(\bar{B}_r)$, moving it to the right hand side we may assume that $\rho=0$.

 Let $\chi \in C_c^{\infty}(B_{\frac{r}{4}})$ be a cut-off function
satisfying $\chi \equiv 1$ in $B_{\frac{r}{8}}$ and $0\leq \chi\leq 1$. Define $v(x):= \chi(x) w(x)\in L^\infty(\Rn)$. For $x\in B_{\frac{r}{4}}$, we write
\begin{align}\label{E3.2B}
L_K v(x) &= \int_{B_1} (v(x) - v(x+z)) K(z) \dz - \int_{B_1^c} v(x+z) K(z)\dz  \nonumber
\\
&=\chi(x)L_K w(x) + \int_{B_1^c} w(x+z)[\chi(x) - \chi(x+z)] K(z)\dz\nonumber
\\
&\qquad + \int_{B_1} w(x+z)[\chi(x) - \chi(x+z)] K(z)\dz\nonumber
\\
&=\chi(x)L_K w(x) + \int_{B_{\frac{r}{2}}^c} w(x+z)[\chi(x) - \chi(x+z)] K(z)\dz \nonumber
\\
&\qquad\quad + \int_{B_{\frac{r}{2}}} w(x+z)[\chi(x) - \chi(x+z)] K(z)\dz\nonumber
\\
&:= F(x).
\end{align}
We claim that 
\begin{equation}\label{E3.2C}
\sup_{B_{\frac{r}{4}}} |F|<\infty,
\end{equation}
which will allow us to use  Theorem~\ref{T2.2} to conclude our result. Applying \eqref{E3.2A} (recall that $\rho=0$), we see that
\begin{align*}
|L_K w(x)| = \frac{1}{|h|^\upkappa}|L_K u(x+eh) - L_K u(x)|
&= \frac{1}{|h|^\upkappa}|f(x+eh) - f (x)|\leq [f]_{C^{0, \upkappa}(B_r)}.
\end{align*}
For $x\in B_{\frac{r}{4}}$, $|x+z|\geq \frac{r}{2}-\frac{r}{4}\geq \frac{r}{4}$ for $|z|\geq \frac{r}{2}$, the middle term in the expression of $F(x)$
can be computed as
\begin{align*}
\chi(x)\left| \int_{B_{\frac{r}{2}}^c} w(x+z) K(z)\dz \right|
&= \frac{\chi(x)}{|h|^{\upkappa}}\left| \int_{B_{\frac{r}{2}}^c} u(x+eh+z) K(z)\dz - \int_{B_{\frac{r}{2}}^c} u(x+z) K(z)\dz\right|
\\
&\leq \frac{1}{|h|^{\upkappa}}\left|\int_{B_{\frac{r}{2}}^c (x+eh)} u(z) K(x+eh-z)\dz - \int_{B_{\frac{r}{2}}^c(x)} u(z) K(x-z)\dz\right|
\\
&= \frac{1}{|h|^{\upkappa}} \Bigg|\int_{B_{\frac{r}{2}}^c (x+eh)} u(z) \left[ K(x+eh-z) - K(x-z) \right]\dz 
\\
&\qquad \qquad+ \int_{B_{\frac{r}{2}}^c (x+eh)} u(z) K(x-z) \dz- \int_{B_{\frac{r}{2}}^c(x)} u(z) K(x-z)\dz\Bigg|
\\
&\leq \frac{1}{|h|^{\upkappa}} \Bigg[C |h|\int_{B_{\frac{r}{2}}^c (x+eh)}   \frac{1}{|x+e h-z|^{n+1}} \dz + C\int_{B_{\frac{r}{2}}^c (x+h) \triangle B_{\frac{r}{2}}^c(x)} u(z) K(x-z) \dz\Bigg]
\\
&\leq \frac{1}{|h|^{\upkappa}} \Bigg[C |h| + C\int_{B_{\frac{r}{2}+|h|} (x) \setminus B_{\frac{r}{2}-|h|}(x)} \frac{1}{|x-z|^{n}} \dz\Bigg]
\\
&\leq C|h|^{1-\upkappa} +\frac{C}{|h|^{\upkappa}} \log \left| \frac{\frac{r}{2}+|h|}{\frac{r}{2}-|h|} \right|
\\
&\leq C|h|^{1-\upkappa},
 \end{align*}
where we have used $u \in L^{\infty} (\Rn)$, \eqref{kernel}-\eqref{kernel-der} and in the fifth line the fact 
$$B_{\frac{r}{2}}^c (x+eh) \triangle B_{\frac{r}{2}}^c(x) \subset B_{\frac{r}{2}+|h|} (x) \setminus B_{\frac{r}{2}-|h|}(x).$$
Using Lipschitz regularity of $\chi$ and $w \in L^{\infty} (B_{\frac{3r}{4}})$, we have
$$\sup_{x\in B_{\frac{r}{4}}}\int_{B_{\frac{r}{2}}} \frac{w(x+z)[\chi(x) - \chi(x+z)]}{|z|^n} \dz \leq C\int_{B_{\frac{r}{2}}}|z|^{-n+1}\dz < \infty.$$
Gathering these estimates we have the claim \eqref{E3.2C}.

Therefore, by \eqref{E3.2B} and Theorem~\ref{T2.2},  we obtain for all $x,y \in B_{r/16}$ that
$$|w(x) - w(y)| \leq C\left(\norm{w}_{L^{\infty}(B_{\frac{r}{2}})} + \norm{F}_{L^\infty(B_{\frac{r}{4}})}+ \int_{\Rn} \frac{|w(z)|}{(1+|z|)^n}\dz\right) |\log|x-y||^{-\beta}.$$
Since {$ \supp v\Subset B_{\frac{r}{4}}$} and $\norm{w}_{L^{\infty}(B_{\frac{r}{4}})}\leq \norm{u}_{C^{0, \upkappa}(B_r)}$, our desired result follows.
\end{proof}
Combining Lemma~\ref{L3.2} and the approach of \cite[Lemma~5.6]{CC95} we obtain a logarithmic improvement of the H\"older regularity.
\begin{lem}\label{L3.3}
Suppose that $r\in (0, \frac{1}{2})$. Let  $u \in C(\bar{B}_{2r}) \cap L^\infty(\mathbb{R}^n) \cap L^1_0(\mathbb{R}^n)$ be a viscosity solution to \eqref{E3.1} in $B_{2r}$. Suppose that $\rho, f \in C^{0, \upkappa}(\bar{B}_{2r})$ with $\rho \geq 0$ for some 
$\upkappa \in (0, 1)$. Then for some $\beta=\beta(n,\kappa)\in (0, 1)$ and $C=C(n, r, \kappa, \norm{\rho}_{C^{0, \upkappa}})>0$ we have
\begin{equation*}
|u(x) - u(y)| \leq C\left(\norm{u}_\infty +  [f]_{C^{0, \upkappa}(B_r)}\right) \frac{|x-y|^\upkappa}{|\log|x-y||^{\beta}}
\quad \text{for}\; x, y\in B_{\frac{r}{64}}.
\end{equation*}
\end{lem}

\begin{proof}
Dividing $u$ by $\norm{u}_\infty +  [f]_{C^{0, \upkappa}(B_r)}$, we may assume that $\norm{u}_\infty +  [f]_{C^{0, \upkappa}(B_r)}\leq 1$. Moreover,
by Theorem~\ref{T3.1} we have $u\in C^{0, \upkappa}(\bar{B}_r)$. Choose $\beta$ from Lemma~\ref{L3.2} and define $g(t)=\frac{t^{\upkappa-1}}{|\log|t||^\beta}$. Let $t_\beta\in (0, \frac{3}{4}]$ be such that $g$ is monotonically decreasing in $(0, t_\beta]$. Moreover, setting
$C_\beta=\max\{1,\frac{1}{g(t_\beta)}\max_{[t_\beta, \frac{3}{4}]} g(t)\}$, we  see that
$$g(t)\leq C_\beta g(s)\quad \text{for} \; 0<s\leq t\leq \frac{3}{4}, \; s\in (0, t_\beta].$$
Given a unit vector $e$, it is enough to bound $|u(x+eh)-u(x)|$ for $x\in B_{\frac{r}{64}}$, $|h|\leq \frac{r}{64}\wedge t_\beta$, uniformly in $e$.
Let $m \geq 0$ be the (smallest) integer satisfying 
$|x+e2^m h| \leq \frac{r}{32} < |x+e2^{m+1} h|$ and define $\tau_0 = 2^m h$. Then, $|\tau_0| - |x| \leq \frac{r}{32} \leq 2|\tau_0| +|x|$ and $|x| < \frac{r}{64} $, implying
\begin{equation}\label{T3.3A}
\frac{r}{128}  < |\tau_0| < \frac{r}{16} .\end{equation}
Define $\ell(t) := u(x+et)-u(x)$. Using Lemma \ref{L3.2} we get
$$|\ell(\tau_0) - 2 \ell(\frac{\tau_0}{2})| = |u(x+ e\tau_0) +u(x) - 2u(x+\frac{e\tau_0}{2})| \leq C \left| \frac{\tau_0}{2} \right|^\upkappa \left|\log |\frac{\tau_0}{2}| \right|^{-\beta}$$
Similarly,
$$|2\ell(\frac{\tau_0}{2}) - 2^2 \ell(\frac{\tau_0}{2^2})|  \leq 2C \left| \frac{\tau_0}{2^2} \right|^\upkappa \left|\log |\frac{\tau_0}{2^2}| \right|^{-\beta},$$
and for $j=1,2 \dots , m$, we get
\begin{align*}
|2^{j-1}\ell(\frac{\tau_0}{2^{j-1}}) - 2^j \ell(\frac{\tau_0}{2^j})| &\leq 2^{j-1 }C  |\frac{\tau_0}{2^j}|^{\upkappa}{\left|\log|{\frac{\tau_0}{2^j}|} \right|^{-\beta}} 
\\
&= 2^{j(1-\upkappa) -1} |\tau_0|^\upkappa {\left|\log|{\frac{\tau_0}{2^j}|} \right|^{-\beta}}.
\end{align*}
Applying triangle inequality, we compute
\begin{align*}
|\ell(\tau_0) - 2^m\ell(\frac{\tau_0}{2^m})| &\leq C |\tau_0|^\upkappa \sum_{j=1}^{m} 2^{j(1-\upkappa) -1}  {\left|\log|{\frac{\tau_0}{2^j}|}\right|^{-\beta}}
\\
& \leq C2^{1-\upkappa}  |\tau_0|^\upkappa \sum_{j=1}^{m} \int\limits_{2^{-j}}^{2^{-j+1}}\frac{t^{\upkappa -2}}{\left|\log(\tau_0 t)\right|^{\beta}}\dt
\\
&\leq C  |\tau_0|^\upkappa  \int\limits_{2^{-m}}^{1}\frac{t^{\upkappa -2}}{\left|\log(\tau_0 t)\right|^{\beta}}\dt
=C  |\tau_0|  \int\limits_{|\tau_0|2^{-m} }^{|\tau_0|}\frac{t^{\upkappa -2}}{\left|\log( t)\right|^{\beta}}\dt.
\end{align*}
Therefore, by substituting $2^{-m}\tau_0=h$ we have
\begin{align}\label{AB001}
|\ell(h)| &\leq 2^{-m} |\ell(\tau_0)| + C |\tau_0|2^{-m} \int\limits_{|h|}^{|\tau_0|}\frac{t^{\upkappa -2}}{\left|\log( t)\right|^{\beta}}\dt\nonumber
\\
&\leq C\frac{|h|}{|\tau_0|} + C |h| \int\limits_{|h|}^{|\tau_0|}\frac{t^{\upkappa -2}}{\left|\log( t)\right|^{\beta}}\dt,
\end{align}
where we have used that $u \in L^{\infty}(\Rn)$. Let $\delta< |\tau_0|$ be such that $\frac{\beta}{1-\upkappa}\frac{1}{|\log\delta|}<\frac{1}{2}$. Then, applying integration by parts, we see that
\begin{align*}
\mathbbm{1}_{[h, \infty)}(\delta) \int\limits_{|h|}^{\delta}\frac{t^{\upkappa -2}}{\left|\log( t)\right|^{\beta}}\dt &= \mathbbm{1}_{[h, \infty)}(\delta) \left[\frac{t^{\upkappa-1}}{\upkappa-1}\frac{1}{\left|\log( t)\right|^{\beta}}\right]^{t=\delta}_{t=h}
+ \mathbbm{1}_{[h, \infty)}(\delta) \frac{\beta}{1-\upkappa}\int\limits_{|h|}^{\delta}\frac{t^{\upkappa -2}}{\left|\log( t)\right|^{1+\beta}}\dt
\\
&\leq \frac{h^{\upkappa-1}}{1-\upkappa}\frac{1}{\left|\log(h)\right|^{\beta}}
+ \mathbbm{1}_{[h, \infty)}(\delta) \frac{\beta}{1-\upkappa}\frac{1}{|\log\delta|}\int\limits_{|h|}^{\delta}\frac{t^{\upkappa -2}}{\left|\log( t)\right|^{\beta}}\dt,
\end{align*}
giving us
$$\mathbbm{1}_{[h, \infty)}(\delta) \int\limits_{|h|}^{\delta}\frac{t^{\upkappa -2}}{\left|\log( t)\right|^{\beta}}\dt\leq \frac{2}{1-\upkappa}\frac{h^{\upkappa-1}}{\left|\log(h)\right|^{\beta}}.$$
Inserting it in \eqref{AB001} we conclude
\begin{align*}
|\ell(h)|\leq C\frac{|h|}{|\tau_0|} + \frac{2C}{(1-\upkappa)}\frac{h^{\upkappa}}{\left|\log(h)\right|^{\beta}} + C |h| \int\limits_{\delta}^{|\tau_0|}\frac{t^{\upkappa -2}}{\left|\log( t)\right|^{\beta}}\dt\leq C_1 \frac{h^{\upkappa}}{\left|\log(h)\right|^{\beta}},
\end{align*}
where the constant $C_1$ depends on $|\tau_0|$ and $\delta$. This completes the proof.
\end{proof}

Finally, we remove the hypothesis of global boundedness from the solution to obtain
\begin{thm}\label{T3.4}
Suppose $r\in (0, \frac{1}{2})$. Let  $u \in C(\bar{B}_{2r}) \cap L^\infty(B_{3r}) \cap L^1_0(\mathbb{R}^n)$ be a viscosity solution to \eqref{E3.1} in $B_{2r}$. Suppose that $\rho, f \in C^{0, \upkappa}(\bar{B}_{2r})$ with $\rho \geq 0$ for some 
$\upkappa \in (0, 1)$. Then for some $\beta=\beta(n,\kappa)\in (0, 1)$ and $C=C(n, r, \kappa, \norm{\rho}_{C^{0, \upkappa}})>0$ we have
\begin{equation*}
|u(x) - u(y)| \leq C\left(\norm{u}_{L^\infty(B_{3r})} +  [f]_{C^{0, \upkappa}(B_r)} + \norm{u}_{L^1_0}\right) \frac{|x-y|^\upkappa}{|\log|x-y||^{\beta}}
\quad \text{for}\; x, y\in B_{\frac{r}{64}}.
\end{equation*}
\end{thm}
\begin{proof}
The proof is quite standard which uses a cut-off function to reduce the problem in the set-up of Lemma~\ref{L3.3}. Firstly, dividing both sides of \eqref{E3.1} by $\norm{u}_{L^\infty(B_{3r})} +  [f]_{C^{0, \upkappa}(B_r)} + \norm{u}_{L^1_0}$, we may assume that
$$
\norm{u}_{L^\infty(B_{3r})} +  [f]_{C^{0, \upkappa}(B_r)} + \norm{u}_{L^1_0}\leq 1.
$$
Now, consider a cut-off function $\chi\in C^\infty_c(B_{\frac{5r}{2}})$, $0\leq \chi\leq 1$ and $\chi=1$ in $B_{\frac{9r}{4}}$. Letting $w(x)=u(x)\chi(x)$, we observe from \eqref{E3.1} that
\begin{equation}\label{ET3.4A}
L_K w + \rho(x) w= f + \underbrace{\int_{|z|\geq \frac{r}{4}}(u(x+z)(1-\chi(x+z))) K(z)\dz}_{:=G(x)}\quad \text{in}\; B_{2r}.
\end{equation}
Now the proof will follow from Lemma~\ref{L3.3} and \eqref{ET3.4A}, once we show that
\begin{equation}\label{ET3.4B}
\sup_{B_{2r}}|G(x)| + \sup_{x\neq y, x, y\in B_{2r}} \frac{|G(x)-G(y)|}{|x-y|}\leq C
\end{equation}
for some constant $C$ dependent on $n, \kappa$ and $r$. 

Write
$$G(x)=\int_{|x-z|\geq \frac{r}{4}}(u(z)(1-\chi(z))) K(x-z)\dz.$$
Since $|x-z|\geq C_1(r+|z|)$ for $|x-z|\geq \frac{r}{4}$, $x\in B_{2r}$, and some universal constant $C_1$, from \eqref{kernel} and the fact $\norm{u}_{L^1_0}\leq 1$ it follows that
$$\sup_{x\in B_{2r}}|G(x)|\leq C,$$
where $C$ depends on $n, \kappa$ and $r$. In view of the above estimate, it is enough to show the Lipschitz property when $|x-y|\leq \frac{r}{10}$. Let $\xi=u(1-\chi)$ and write
\begin{align*}
|G(x)-G(y)|&\leq \underbrace{\left|\int_{B^c_{\frac{r}{4}}(x)} \xi(z)K(x-z)\dz-\int_{B^c_{\frac{r}{4}}(y)} \xi(z)K(x-z)\dz\right|}_{:=J_1}
\\
&\quad  + \underbrace{\left|\int_{B^c_{\frac{r}{4}}(y)} \xi(z)K(x-z)\dz-\int_{B^c_{\frac{r}{4}}(y)} \xi(z)K(y-z)\dz\right|}_{:=J_2}.
\end{align*}
Since, for $x, y\in B_{2r}$,
$$B^c_{\frac{r}{4}}(x)\triangle B^c_{\frac{r}{4}}(y)=B_{\frac{r}{4}}(x)\triangle B_{\frac{r}{4}}(y)\subset B^c_{\frac{r}{10}}(x)\cap B_{3r},$$
we have 
$$J_1\leq \kappa \norm{u}_{L^\infty(B_{3r})} r^{-n} {|B_{\frac{r}{4}}(x)\triangle B_{\frac{r}{4}}(y)|}\leq C |x-y|$$
for some constant $C$, dependent on $n, \kappa$ and $r$. On the other hand, using \eqref{kernel-der}, we see that
$$|K(x-z)-K(y-z)|\leq C \frac{|x-y|}{|y-z|^{n+1}}$$
for $|y-z|\geq \frac{r}{4}$. Since $\norm{u}_{L^1_0}\leq 1$, we get $J_2\leq C |x-y|$. Combining these estimates we thus have \eqref{ET3.4B}. Hence the proof.
\end{proof}
\begin{proof}[Proof of Theorem~\ref{T1.2}]
Note that applying the same trick of cut-off function, as in Theorem~\ref{T3.4}, 
we can conclude from Theorem~\ref{T2.2} that $u\in C(\Omega)$.
Since $\rho_n>0$ for $n\geq 2$, proof follows by applying a covering argument together with Proposition~\ref{P2.4} and Theorem~\ref{T3.4}.
\end{proof}

\section{Harnack inequality}\label{Harnack}
In this section, we prove the Harnack inequality for operators of the form $L_K u + \rho u$. Our analysis requires $K$ to satisfy only \eqref{kernel}, without relying on \eqref{kernel-der}. The core strategy of the proof follows \cite{DRV22} (see also \cite[Section~3.3]{FRRO}). 
To account for the operator's lack of scale invariance and the non-integrability of the kernel's tail, we introduce suitable modifications into the proof.

We first prove a weak Harnack inequality for nonnegative supersolutions.
\begin{thm}[Weak Harnack inequality]\label{weak-har} 
Let $u \in L_0^1(\Rn) \cap LSC(B_{r})$, $r\in (0, \frac{1}{4}]$, and $\rho\in L^\infty(B_{r})$ satisfy

\begin{gather*}
L_K u + \rho u \geq -C_0 \quad  \text{in}\; B_{r},
\\[1mm]
u \geq 0 \quad  \text{in}\;\Rn,
\end{gather*}
in the viscosity sense, for some $C_0\geq 0.$ Then there exist $C$, dependent on $n,\kappa, r$ and $\norm{\rho}_{L^\infty(B_{r})}$, such that 
$$\|u\|_{L_0^1(\Rn)}\leq C\left(\inf_{B_{\frac{r}{2}}}u+C_0\right).$$
\end{thm}

\begin{proof}
First suppose that there is a point $x_0\in B_r$ satisfying $u(x_0)=0$. In this case, we can use $\varphi\equiv 0$ as a test function which would touch
$u$ from below at the point $x_0$ in a ball $B_\delta(x_0)$ for all sufficiently small $\delta>0$. Thus, from the definition of viscosity solution
we obtain
$$-\int_{|x_0-z|\geq\delta} u(z)K(x_0-z)\dz \geq -C_0,$$
giving us
\begin{align*}
C_0 + \inf_{B_{\frac{r}{2}}} u\geq C_0\geq \int_{|x_0-z|\geq\delta} u(z)K(x_0-z)\, \dz &\geq \kappa^{-1}\int_{|x_0-z|\geq\delta} \frac{u(z)}{1+|x_0-z|^n}\dz
\\
&\geq C \int_{|x_0-z|\geq\delta} \frac{u(z)}{1+|z|^n}\dz.
\end{align*}
Letting $\delta\to 0$, we get the result.

Next we assume $u>0$ in $B_r$.
Let $\eta\in C_c^\infty(B_{\frac{3r}{4}})$ be such that $\eta\in [0,1]$ and $\eta\equiv 1$ in $B_{\frac{r}{2}}$. Define
 $$t=\max\{\tau \geq 0: \tau\eta\leq u\;\text{in}\;B_{r}\}.$$ 
 Since $\eta=1$ in $B_{\frac{r}{2}}$, $t$ has to be finite.
Since $(u-t\eta)$ is lower semicontinuous in $\bar{B}_{\frac{3r}{4}},$
there exists $x_0\in B_{\frac{3r}{4}}$ such that $t\eta(x_0)=u(x_0)$. In other words, $t\eta$ touches $u$ from below at $x_0$ in 
$B_{r}$. For $\delta\ll\frac{r}{4}$, define
   $$
\varphi_\delta(x)=\left\{\begin{array}{ll}
t\eta(x) & \text{if}\; x\in B_\delta(x_0),
\\[2mm]
u(x) & \text{otherwise}.\;
\end{array}
\right.
$$
By the definition of viscosity solution 
$$
L_K\varphi_\delta(x_0)+\rho(x_0)\varphi_\delta(x_0)\geq -C_0,
$$
 which implies
\begin{equation}\label{ET4.1A}
\int_{|z|\geq \delta}(u(x_0)\mathbbm{1}_{B_1}(z) - u(x_0+z) )K(z)\dz + t \int_{|z|< \delta} (\eta(x_0)\mathbbm{1}_{B_1}(z) -\eta(x_0+z))K(z)\dz\geq -C_0 - t\norm{\rho}_{L^\infty(B_{r})}.
\end{equation}
Moreover, since $u-t\eta\geq 0$ in $\Rn$ and $(u-t\eta)(x_0)=0$, we get
\begin{align*}
   & \int_{|z|\geq \delta}(u(x_0)\mathbbm{1}_{B_1}(z) - u(x_0+z) )K(z)\dz 
 \\
 &\quad=
    \int_{|z|\geq \delta}\left((u-t\eta)(x_0)\mathbbm{1}_{B_1}(z) -(u-t\eta)(x_0+z)\right) K(z)\dz
\\
&\qquad +t\int_{|z|\geq \delta}(\eta(x_0)\mathbbm{1}_{B_1} (z)- \eta(x_0+z)) K(z)\dz
\\
&\quad\leq -\frac{1}{\kappa}\int_{|z|\geq \delta}\frac{(u-t\eta)(x_0+z)}{1+|z|^n}\dz+  t \kappa \int_{\Rn}\frac{|\eta(x_0)\mathbbm{1}_{B_1} (z) -\eta(x_0+z) |}{|z|^n}\dz
\\
&\quad \leq - \frac{1}{\kappa}\int_{|z|\geq \delta}\frac{u(x_0+z)}{1+|z|^n}\dz + tC_\eta
\end{align*}
 for some $C_\eta>0$. Substituting this into \eqref{ET4.1A} and letting $\delta\to 0$, we obtain
$$\frac{1}{\kappa}\int_{\Rn}\frac{u(x_0+z)}{1+|z|^n}\dz \leq C_0 + t C_\eta + t\norm{\rho}_{L^\infty(B_{r})}.$$ 
Observe that $t\leq \inf_{B_{\frac{r}{2}}}u$. Hence
\begin{align*}
C_0+\left(C_\eta + \norm{\rho}_{L^\infty(B_{r})}\right)\inf_{B_{\frac{r}{2}}}u &\geq \frac{1}{\kappa} \int_{\Rn}\frac{u(x_0+z)}{1+|z|^n}\dz
\\
&= \frac{1}{\kappa} \int_{\Rn}\frac{u(z)}{1+|z-x_0|^n}\dz
\\
&\geq C \int_{\Rn}\frac{u(z)}{1+|z|^n}\dz.
\end{align*}
This completes the proof.
\end{proof}

We next prove the other half of the Harnack inequality. The following fact will be used at several occasions in the proof below. If 
$w\in C^{0,\gamma}$ for some $\gamma\in (0, 1)$,
and $u$ is a viscosity sub/super solution, then $L_K(u+w)=L_K u + L_K w$ makes sense in the viscosity sense. 
\begin{thm}[Half-Harnack inequality]\label{half-har}
Let $u \in L_0^1(\Rn) \cap USC(B_{r}), r\in (0, \frac{1}{4})$, satisfy
$$
\left\{\begin{array}{ll}
L_K u + \rho u \leq C_0 & \text{in}\; B_{r},
\\[2mm]
u\geq 0 & \text{in}\;\Rn,
\end{array}
\right.
$$
in the viscosity sense, for some $C_0\geq 0$. Also, assume that $\rho\geq 0$.
Then there exists $C=C(n, r, \kappa)>0$ such that $$\sup_{B_{\frac{r}{8}}} u\leq C\left(\|u\|_{L_0^1(\Rn)}+C_0\right).$$
\end{thm}

\begin{proof} 
Let us first argue that we may take $C_0=0$ and $\rho=0$. Note that since $\rho, u\geq 0$, $L_Ku +\rho u\leq C_0$ implies
$L_Ku \leq C_0$. Hence we may set $\rho=0$. Again,
if $C_0>0$, define $\tilde{u}:=u - C_1 C_0 \eta,$ for some $C_1>0$ and $\eta\in C_c^\infty(B_{2r})$, $\eta \equiv 1 \text{ in } B_{r}$ and  $0\leq \eta\leq 1$. Note that 
$$L_K\eta = \int_{B_1} (1-\eta(x+z))K(z)\dz \geq \inf_{B_{r}}\int_{B_1} (1-\eta(x+z))K(z)\dz:=c > 0$$
 in $ B_{r}$. Thus
$$L_K\tilde{u} \leq C_0 - C_1 C_0 L_K \eta  \leq 0 \quad\text{in } B_{r},$$
provided we set $C_1=\frac{1}{c}$. Therefore, it suffices to prove the result for $\tilde{u}$ since
$$
\sup_{B_{\frac{r}{2}}} u \leq \sup_{B_{\frac{r}{2}}} \tilde u  + C_1 C_0 \leq C\left(\|u\|_{L_0^1(\Rn)} + C_1 C_0 \|\eta\|_{L_0^1(\Rn)} +C_0 \right),$$
would imply the desired estimate for $u$. From now on,  we denote $\tilde u$ by $u$ and we have 
\begin{equation}\label{ET4.2A}
L_K u\leq 0.
\end{equation}
Note that $0$ is also a solution to the above equation and therefore, $u_+=\max\{u, 0\}$ is a viscosity subsolution to \eqref{ET4.2A}. Hence we may also 
assume that $u\geq 0$.
If $ u \equiv 0$, the conclusion holds. So we consider $u\neq 0$.
We may assume, upon dividing by a constant, that $\norm{u}_{L_0^1(\Rn)} = 1$.
Hence it is sufficient to prove that if $u \in  L_0^1(\Rn) \cap USC(B_{r})$ is a viscosity solution to
\begin{equation}\label{ET4.2B}
\left\{\begin{array}{ll}
L_K u \leq 0 & \text{in}\; B_{r},
\\[2mm]
u\geq 0 & \text{in}\;\Rn,
\end{array}
\right. \quad \text{ and } \quad \norm{u}_{L_0^1(\Rn)} = 1,
\end{equation}
then
\begin{equation}\label{ET4.2C}
\sup_{B_{\frac{r}{8}}} u\leq C
\end{equation}
for some $C$ depending only on $n, \kappa$ and $r$.

Now, we begin with a scaling argument. Fix $x_0\in B_{\frac{r}{4}}$ and let $ 0<\theta<\frac{r}{8}$. Define 
$$w(x)=u(x_0 + \theta x)\quad \text{and}\quad K_\theta(z)=\theta^nK(\theta z).$$
For $x \in B_1$
\begin{align}\label{ET4.2D}
L_Ku(x_0+\theta x) &= \int_{\Rn} \left(u(x_0+\theta x)\mathbbm 1_{B_1}(z) - u(x_0+\theta x+z)\right) K(z) \dz \nonumber
\\
&=\int_{\Rn} \left(u(x_0+\theta x)\mathbbm 1_{B_{1/\theta}}(z) - u(x_0+\theta (x+z)) \right) K_\theta(z) \dz \nonumber
\\
&= \int_{\Rn} \left(w(x)\mathbbm 1_{B_{1/\theta}}(z) - w(x+z) \right) K_\theta(z) \dz \nonumber
\\
&= \int_{\Rn} \left(w(x)\mathbbm 1_{B_{1}}(z) - w(x+z) \right) K_\theta(z) \dz
+ w(x) \int_{\Rn}\left(\mathbbm 1_{B_{1/\theta}}(z) - \mathbbm 1_{B_1}(z)\right)K_\theta(z)\dz \nonumber
\\
&= L_{K_\theta} w(x) + \widehat{\rho} w(x),
\end{align}
where 
\begin{align*}
\widehat{\rho}= \int_{\Rn}\left(\mathbbm{1}_{B_{1/\theta}}(z) - \mathbbm{1}_{B_1}(z)\right)K_\theta(z)\dz 
=\int_{\{1\leq|z|<1/\theta\}} K_\theta(z)\dz >0.
\end{align*}
Using \eqref{ET4.2B} and \eqref{ET4.2D}, we then obtain for $x_0 \in B_{\frac{r}{4}}$ that 
$$
 L_{K_\theta} w \leq 0\quad \text{in}\; B_1,
$$
and the scaled kernel $K_\theta$ satisfies \eqref{kernel} with the same constant $\kappa$.

\textbf{Claim:} There exist constants $ \delta, c_0 \in (0,1)$, depending only on $n$ and $\kappa$, such that if $u(x_0) \geq M$ for some $x_0 \in B_{\frac{r}{4}}$ and $r_0 =(c_0 M)^{-\frac{1}{n}}< \frac{r}{16}$, then
$$ \sup\limits_{B_{r_0}(x_0)}u > (1+\delta) M.$$
\\
To prove the claim we argue by contradiction. Suppose that
$$ \sup\limits_{B_{r_0}(x_0)}u \leq (1+\delta)M.$$
Define
$$ w(x):=u(x_0+2r_0x)\quad \text{for}\; x\in\Rn. $$
For $x\in B_{\frac{1}{2}}$, we have $ |2 r_0 x|<r_0$, implying
$ w(x)\leq(1+\delta)M$.
Since $ 2 r_0 < \frac{r}{8}$, preceding scaling argument applies with $ \theta = 2 r_0$, and we have
\begin{equation}\label{ET4.2E}
 L_{K_\theta} w \leq 0 \quad \text{in } B_{\frac{1}{4}}.
\end{equation}
Choose a smooth cutoff function $ \eta \in C_c^\infty(B_{3/4}),  0\leq \eta\leq 1, $ and $\eta = 1$ in $B_{\frac{1}{2}}$. It is easily seen that
\begin{equation}\label{ET4.2F1}
 L_{K_\theta}\eta(x)\geq 0 \quad\text{for all }x\in B_{\frac{1}{4}}.
\end{equation}
Define
$$ v(x):=(1+\delta)M\eta(x)-w(x).$$
Using \eqref{ET4.2E} and \eqref{ET4.2F1}, we obtain
$$L_{K_\theta} v = (1+\delta)M L_{K_\theta}\eta-L_{K_\theta} w \geq 0 \qquad\text{in }B_{\frac{1}{4}}.$$
Letting
$ v_+:=\max\{v,0\},\, v_-:=\max\{-v,0\}$, we have
$ v=v_+-v_-$, and since $v\geq0$ in $B_{\frac{1}{2}}$, we have
\begin{equation}\label{ET4.2F}
v_- \equiv 0 \qquad \text{in }B_{\frac{1}{2}}.
\end{equation}
Thus, for $x\in B_{\frac{1}{4}}$, the function $v_-$ vanishes in a neighbourhood of $x$. Consequently,
$$L_{K_\theta} v_-(x) = -\int_{\Rn} v_-(x+z) K_\theta(z)\dz.$$
Moreover, since $v=(1+\delta)M\eta-w$ and $\eta\geq0$, we have $v_-\leq w$. From \eqref{ET4.2F} and \eqref{kernel}, we then find
\begin{align*}
L_{K_\theta} v_-(x) &\geq -\kappa \int_{|x+z|\geq\frac{1}{2}} \frac{u(x_0+2r_0(x+z))}{|z|^{n}}\,\dz
\\
&= -\kappa \int_{|z-x_0|\geq r_0} \frac{u(z)} {|z-(x_0+2r_0x)|^{n}} \dz.
\end{align*}
Since for $|z-x_0|\geq r_0$, we have 
$$|z-(x_0+2r_0x)| \geq |z-x_0|-2r_0|x| \geq |z-x_0|-\frac{r_0}{2} \geq \frac{1}{2}|z-x_0|,$$
leading to
$$ \frac{\kappa}{|z-(x_0+2r_0x)|^{n}} \leq \frac{C_1}{|z-x_0|^{n}}.$$
Next observe that
\begin{equation}\label{ET4.2G}
\frac{1}{|z-x_0|^{n}} \leq \frac{C_2 r_0^{-n}}{(1+|z|)^{n}}, \quad \text{whenever }|z-x_0|\geq r_0.
\end{equation}
Indeed, if $|z|\leq1$, then
$$ |z-x_0|^{-n} \leq r_0^{-n} \leq \frac{2^nr_0^{-n}}{(1+|z|)^{n}},$$
whereas if $|z|>1$, then, since $x_0\in B_{\frac{r}{4}}$, we have $ |z-x_0| \geq |z|-|x_0| \geq |z|-\frac{1}{16} \geq \frac{1}{2}|z|\geq r_0(1+|z|)$,
giving
$$ |z-x_0|^{-n} \leq  \frac{r_0^{-n}}{(1+|z|)^{n}}.
$$
This gives us \eqref{ET4.2G}.
Using \eqref{ET4.2B}, we get
$$ L_{K_\theta} v_-(x) \geq -C_3 r_0^{-n} \int_{\Rn} \frac{u(z)}{(1+|z|)^{n}}\dz = -C_3 r_0^{-n} =-C_3 c_0 M \quad \text{ in }B_{\frac{1}{4}}, $$
where the constant $C_3$ depends only on $n$ and $\kappa$. Hence
$$
L_{K_\theta} v_+ = L_{K_\theta} v + L_{K_\theta} v_- \geq -C_3 c_0 M \quad \text{in }B_{\frac{1}{4}}
$$
 in the viscosity sense. Since $v_+ (0) = (1+\delta)M-u(x_0)\leq \delta M$, applying  Theorem ~\ref{weak-har} to $v_+$, we obtain
\begin{equation}\label{ET4.2H}
\int_{B_{\frac{1}{2}}}v_+(x)\dx \leq C\|v_+\|_{L_0^1(\Rn)} \leq C\left(\inf_{B_{\frac{1}{8}}}v_+ + c_0 M \right)  \leq C\left(v_+ (0) + c_0 M \right) \leq C(c_0+\delta)M.
\end{equation}
Choosing $c_0>0$ and $\delta>0$ sufficiently small in \eqref{ET4.2H} we then get
$$ \int_{B_{\frac{1}{2}}}v_+(x) \dx \leq \frac{M}{2} |B_{\frac{1}{2}}|.$$
Therefore,
\begin{align*}
\int_{B_{\frac{1}{2}}}w(x)\,dx &= (1+\delta)M|B_{\frac{1}{2}}| - \int_{B_{\frac{1}{2}}}v_+(x)\dx
\geq \left( 1+\delta-\frac{1}{2} \right) M |B_{\frac{1}{2}}|
\geq \frac M2|B_{\frac{1}{2}}|,
\end{align*}
implying
$$
\fint_{B_{\frac{1}{2}}}w(x) \dx \geq \frac{M}{2}.
$$
Recalling  $ w(x)=u(x_0+2 r_0 x)$ and applying the change of variable $ z=x_0+2 r_0 x$, we arrive at
$$ \fint_{B_{r_0}(x_0)}u(z)\dz = \fint_{B_{\frac{1}{2}}}w(x)\dx \geq \frac{M}{2}.$$
Therefore,
\begin{equation}\label{ET4.2I}
\int_{B_{r_0}(x_0)}u(z)\dz \geq \frac{M}{2} |B_{r_0}|.
\end{equation}
Since $ x_0\in B_{\frac{r}{4}} \text{ and } r_0<\frac{r}{16}$, we see that
$$ |z| \leq |z-x_0|+|x_0| < r_0+\frac{r}{4} < \frac{r}{16}+\frac{r}{4} = \frac{5r}{16}<\frac{1}{4}$$
for every $z\in B_{r_0}(x_0)$. In particular, $B_{r_0}(x_0)\subset B_{\frac{1}{4}}$. Hence
$ (1+|z|)^{n}\leq 2^n$ for $z\in B_{r_0}(x_0)$.
Using \eqref{ET4.2B} and \eqref{ET4.2I}, we find
$$1 = \int_{\Rn} \frac{u(z)}{(1+|z|)^{n}}\dz \geq \frac{1}{2^n} \int_{B_{r_0}(x_0)}u(z)\dz \geq \frac{M}{2^{n+1}}|B_{r_0}| = \frac{M}{2^{n+1}} |B_1|r_0^{n}.$$
Since $ r_0^{n} = \frac{1} {c_0M}$, we obtain $ 1 \geq \frac{|B_1|}{2^{n+1}c_0}$. Choosing $c_0>0$ further smaller, if necessary, so that $ c_0<\frac{|B_1|}{2^{n+1}}$, we obtain a contradiction. This proves the claim.

We now proceed to prove \eqref{ET4.2C} using our claim established above.
Suppose, on the contrary, that for some $M_0>0$ sufficiently large, there exists $x_0\in B_{\frac{r}{8}}$ such that $u(x_0)>M_0$. Define
$$ M_j=(1+\delta)^jM_0,\quad r_j=(c_0M_j)^{-1/n},\quad \text{for}\; \; j\geq0.$$
By the claim, there exists $ x_1\in B_{r_0}(x_0)$ such that
$$ u(x_1)>M_1=(1+\delta)M_0.$$ 
Applying the claim again at $x_1$, there exists $ x_2\in B_{r_1}(x_1)$ such that 
$u(x_2)>M_2=(1+\delta)^2M_0$ and so on.
This gives us a sequence $\{x_j\}$ satisfying
$ |x_{j+1}-x_j|<r_j$ and 
\begin{equation}\label{eqitr}
u(x_j)>M_j=(1+\delta)^j M_0.
\end{equation}
Since $r_j =c_0^{-1/n}M_0^{-1/n}(1+\delta)^{-j/n}$, we have
\begin{align*}
\sum_{j=0}^{\infty}r_j &= c_0^{-1/n}M_0^{-1/n} \sum_{j=0}^{\infty}(1+\delta)^{-j/n}
= \frac{c_0^{-1/n}M_0^{-1/n}}{1-(1+\delta)^{-1/n}}.
\end{align*}
Choose $M_0$ sufficiently large so that
\begin{equation}\label{eqsumr}
\sum_{j=0}^{\infty}r_j<\frac{r}{16}.
\end{equation}
Then, using $x_0\in B_{\frac{r}{8}}$ and \eqref{eqsumr}, we obtain
\begin{align*}
|x_j|&\leq |x_0|+\sum_{k=0}^{j-1}|x_{k+1}-x_k|
\\
&<\frac{r}{8}+\sum_{k=0}^{\infty}r_k
\\
&<\frac{r}{8}+\frac{r}{16}<\frac{r}{4}.
\end{align*}
This  shows that $x_j \in B_{\frac{r}{4}}$ for each $j$, and the claim can be applied at every step. However, by \eqref{eqitr},
$$u(x_j) \rightarrow+\infty \qquad \text{as }\;j\to\infty.$$
This contradicts the assumption $ u\in USC(B_r)$. Therefore, there exists a constant $C>0$, depending only on $n$ and $\kappa$, such that
$$ \sup_{B_{\frac{r}{8}}}u\leq C.$$
Hence the proof.
\end{proof}

Combining Theorems~\ref{weak-har} and ~\ref{half-har} we get the Harnack inequality.
\begin{thm}\label{T4.3}
Suppose that the kernel $K$ satisfies \eqref{kernel}.
Let $u \in L_0^1(\Rn) \cap C(B_{r}), r\in (0, \frac{1}{4})$, satisfy
\begin{align*}
-C_0\leq L_K u + \rho u &\leq C_0 \quad  \text{in}\; B_{r},
\\[2mm]
u&\geq 0\quad  \text{in}\; B_{3r},
\end{align*}
in the viscosity sense, for some $C_0\geq 0$. Also, assume that $\rho\in L^\infty(B_{r})$ and $\rho\geq 0$.
Then there exists $C=C(n, r, \kappa, \norm{\rho}_\infty)>0$ such that 
$$\sup_{B_{\frac{r}{8}}} u\leq C\left(\inf_{B_{\frac{r}{8}}} u + C_0 + \int_{B^{c}_{3r}} \frac{u_{-}(z)}{|z|^n}\dz\right).$$
\end{thm}

\begin{proof}
Since $u_{-}\equiv 0$ in $B_{3r}$ we get 
$$
- \kappa\int_{|z|\geq 3r}\frac{u_{-}(x+z)}{|z|^n}\dz-C_0\leq L_K u_+ + \rho u_+ \leq C_0 - \kappa^{-1}\int_{|z|\geq 3r}\frac{u_{-}(x+z)}{|z|^n}\dz
$$
in $B_r$, in the viscosity sense. Since, for $x\in B_r$,
\begin{align*}
\int_{|z|\geq 3r}\frac{u_{-}(x+z)}{|z|^n}\dz=\int_{|x-z|\geq 3r}\frac{u_{-}(z)}{|x-z|^n}\dz\leq \int_{|z|\geq 2r}\frac{u_{-}(z)}{|x-z|^n} \dz
&\leq 2^n\int_{|z|\geq 2r}\frac{u_{-}(z)}{|z|^n} \dz\\
&=
2^n\int_{|z|\geq 3r}\frac{u_{-}(z)}{|z|^n} \dz,
\end{align*}
the proof follows from Theorems~\ref{weak-har} and ~\ref{half-har}.
\end{proof}

\begin{proof}[Proof of Theorem~\ref{T1.3}]
Note that applying the same trick as in Theorem~\ref{T3.4}, 
we can conclude from Theorem~\ref{T2.2} that $u\in C(\Omega)$.
Now the proof follows from Proposition~\ref{P2.4} and Theorem~\ref{T4.3}.
\end{proof}

\bigskip
\subsection*{Acknowledgement}
A. Biswas was partially supported by 
an ANRF-ARG grant ANRF/ARG/ 2025/000019/MS .

\subsection*{Data Availability}  All data generated or analyzed during this study are included in this published article.
\medskip

\noindent{\bf Declarations}
\subsection*{Conflicts of Interest}  The authors declare that they have no conflict of interest to this work.

\end{document}